\documentclass[final,leqno,onefignum,onetabnum]{siamltex1213}

\usepackage{mathrsfs}
\usepackage{amssymb}
\usepackage{amsmath}

\usepackage{soul}

\newcommand{\add}[1]{#1}
\newcommand{\del}[1]{}
\newcommand{\dell}[1]{}

\DeclareSymbolFont{bbold}{U}{bbold}{m}{n}
\DeclareSymbolFontAlphabet{\mathbbold}{bbold}

\newtheorem{remark}[theorem]{Remark}

\newtheorem{assump}{}

\newcounter{assumptions}[section]
\addtocounter{assumptions}{1}

\newcommand{\ito}{It\=o}

\newcommand{\bi}{\begin{itemize}}

\newcommand{\ei}{\end{itemize}}
\newcommand{\ba}{\begin{align*}}
\newcommand{\ea}{\end{align*}}
\newcommand{\be}{\begin{equation}}
\newcommand{\ee}{\end{equation}}

\renewcommand{\l}{\left}
\renewcommand{\r}{\right}

\newcommand{\R}{\mathbb{R}}

\renewcommand{\P}{\mathscr{P}}

\renewcommand{\H}{\mathscr{H}}

\newcommand{\Ptwo}{\mathscr{P}_{2}(\mathbb{R})}
\newcommand{\Htwo}{\H^{2}([0,T];\mathbb{R})}

\newcommand{\Ltwo}{\mathscr{L}^{2}}

\newcommand{\tF}{\tilde{\F}}

\newcommand{\F}{\mathscr{F}}
\newcommand{\G}{\mathscr{G}}
\newcommand{\tmb}[1]{\tilde{\mathbb{#1}}}
\newcommand{\Law}{\mathcal{L}}

\newcommand{\tsg}{\tilde{\sigma}}
\newcommand{\sg}{\sigma}

\newcommand{\mb}{\mathbb}
\newcommand{\mbb}{\mathbb}

\newcommand{\eps}{\varepsilon}

\newcommand{\tdW}{d\tilde{W}}
\newcommand{\tW}{\tilde{W}}

\newcommand{\tZ}{\tilde{Z}}

\newcommand{\hX}{\hat{X}}
\newcommand{\hY}{\hat{Y}}
\newcommand{\hZ}{\hat{Z}}
\newcommand{\htZ}{\hat{\tilde{Z}}}

\renewcommand{\a}{\alpha}
\newcommand{\ha}{\hat{\a}}

\title{Wellposedness of Mean Field Games with Common Noise under a Weak Monotonicity Condition}
\author{Saran Ahuja\thanks{Department of Mathematics, Stanford University, Sloan Hall, Stanford, California 94305 \email{(ssunny@stanford.edu)}.}}
                                       
\begin{document}
\maketitle

\begin{abstract}In this paper, we consider Mean Field Games in the presence of common noise relaxing the usual independence assumption of individual random noise. We assume a simple linear model with terminal cost satisfying a convexity and a weak monotonicity property. Our main result is showing existence and uniqueness of a Mean Field Game solution using the Stochastic Maximum Principle. The uniqueness is a result of a monotonicity property similar to \add{that of Lasry and Lions}\del{those used in other work}. We use the Banach Fixed Point Theorem to establish an existence over a small time duration and show that it can be extended over an arbitrary finite time duration.
\end{abstract}

\begin{keywords} mean field games, common noise, forward-backward stochastic differential equations, stochastic maximum principle. \end{keywords}

\begin{AMS} 93E20, 60H30, 60H10, 49N70, 49J99.\end{AMS}

\pagestyle{myheadings}
\thispagestyle{plain}
\markboth{SARAN AHUJA}{MEAN FIELD GAMES WITH COMMON NOISE}

\section{Introduction}\label{introduction}
A Mean Field Game, or MFG for short, is a model recently proposed by Lasry and Lions in his series of papers \cite{lasry2006, lasry2006_2,lasry2007} and independently by Caines, Huang, and Malham\'e \cite{huang2007}, who named it differently as \textit{Nash Certainty Equivalence}. It describes a limiting model of a stochastic differential game with a large number of players, symmetric cost, and weak interaction. Specifically, each player executes a stochastic control problem whose cost and/or dynamics depend not only on their own state and control but also on other players' states. However, this interaction is weak in a sense that a player feels the effect of other players only through the empirical distribution. Searching for a Nash equilibrium strategy for $N$-player games is known to be intractable when $N$ is large as the dimensionality grows with the number of players. However, by assuming independence of the random noise in the players' state processes, symmetry of the cost functions, and a mean-field interaction, \del{taking the limit as $N \to \infty$ reduces}\add{we can formally take the limit as $N \to \infty$ and reduce} a problem to solving a fully-coupled system of forward-backward Partial Differential Equations (PDEs). The backward one is a Hamilton-Jacobi-Bellman (HJB) equation for the value function for each player while the forward one is the Fokker-Planck (FP) equation for the evolution of the player's distribution. This limiting system is more tractable and one can use its solution to approximate Nash equilibrium strategies of the $N$-player games. Lasry and Lions studied this limiting model extensively as well as rigorously analyzed the limit in some cases \cite{lasry2006, lasry2006_2,lasry2007}. 

Since the introduction of MFG in 2006, the literature in this area has grown tremendously. See \cite{bensoussan2013, gomes2013survey} or for a recent survey and \cite{cardaliaguet2010} for a summary of a series of Lions' lectures given at the Coll\'ege de France. Gu\'eant wrote an introduction to the area highlighting various applications \cite{gueant2011}. See \cite{gueant2009,gueant2011graph,gueant2012} for several other contributions by Gu\'eant on the subject. Gomes, Mohr, and Souza studied the finite state problem, both in discrete time \cite{gomes2010discrete} and continuous time \cite{gomes2012cont}. Carmona and Delarue \cite{Carmona1} approached the MFG problem from a probabilistic point of view. By using a Stochastic Maximum Principle instead of a Dynamic Programming Principle to formulate the control problem, one gets Forward Backward Stochastic Differential Equations (FBSDEs for short) of McKean-Vlasov type instead of coupled HJB-FP equations.

One important assumption in most of the prior work is independence of the random factors in each player's state processes. From this assumption\del{ and the Law of Large Numbers}, the distribution of each player evolves deterministically in the limit. It is this property that plays a key role in reducing the dimension of this complex problem and making it tractable. However, many models in applications do not satisfy this assumption. For instance, in financial applications, any reasonable model which attempts to understand the interactions of a large number of market participants will have to assume that they are exposed to some type of overall \textit{market} randomness. This \textit{common} random factor is applied to all the players, and, as a result, the independence assumption does not hold. See \cite{carmona2013mean} for an explicit example of a linear-quadratic MFG model with common noise used to model inter-bank lending and borrowing.

The presence of common noise clearly adds an extra layer of complexity to the problem as the empirical distribution of players at the limit now evolves stochastically. Following the PDE approach of Lasry and Lions, common noise then turns HJB-FP equations to stochastic HJB-FP equations which are of Forward Backward Stochastic Partial Differential Equation (FBSPDE for short) type. This FBSPDE is clearly more complicated than the FBPDE counterpart in the original MFG. Alternatively, one could abandon the forward backward coupling and reformulate the problem as a single PDE called the \textit{master equation}. While the master equation contains all the information about the MFG model, it requires new theories and tools as it is a second order infinite-dimensional HJB involving derivatives with respect to probability measures. To the best of our knowledge, there is no existence theory for it and most of the discussion so far has been largely formal. See \cite{bensoussan2014master,carmona2014master,gomes2013} for instance. On the other hand\del{s}, as we shall see, the probabilistic approach of Carmona and Delarue can be extended more naturally to accommodate the common noise. The law of the state process which occurs in the McKean-Vlasov FBSDE from the Stochastic Maximum Principle will simply be replaced by its conditional law given the common Brownian motion path.

\add{The goal of this paper is to establish existence and uniqueness for MFG models in the presence of common noise. Similar to Carmona and Delarue \cite{Carmona1} for the case of no common noise, we use a probabilistic approach to MFG based on the Stochastic Maximum Principle and operate under a linear-convex framework. Nonetheless, the existence proof in \cite{Carmona1}, which relies on Schauder fixed point theorem applied to a compact subset of \textit{deterministic} flow of probability measures, does not carry over to the case of common noise. In this case, we are working with a much larger space of \textit{stochastic} flow of probability measures of which we cannot find an invariant compact subset. Instead, we apply Banach fixed point theorem to show existence over small time duration and attempt to extend the solution over arbitrary time. To accomplish this, we introduce the \textit{weak monotonicity} assumption (see \ref{a4}) on the cost function. This condition can be viewed as a stronger version of the \textit{weak mean reverting condition} used in \cite{Carmona1} to show existence for the model without common noise. For the uniqueness, our condition, as the name suggests, is weaker than the monotonicity condition of Lasry and Lions used in \cite{cardaliaguet2010,Carmona1,gomes2013}. As a by product, we extend the uniqueness result for MFG without common noise. More importantly, this condition covers many interesting and practical cases which do not verify Lasry and Lions' monotonicity condition; a notable example includes linear-quadratic models. Intuitively, this condition gives a monotone property to the corresponding McKean-Vlasov FBSDE from the Stochastic Maximum Principle. Similar to the result of Peng and Wu \cite{peng1999} for a classical FBSDE, wellposed-ness of this McKean-Vlasov FBSDE can be shown under such monotone property. For notational simplicity, we select a simple state process and running cost function, but from the result of \cite{peng1999}, our result is expected to hold under a more general setting (see Section \ref{summary}). We would like to note that we do not deal with the $N$-player games in this paper. Here, we focus on the wellposed-ness of the MFG models with common noise.}



\add{In parallel to our work, there has been two recent papers \cite{carmona2014commonnoise,lacker2014translation} which contribute to the general existence theory of MFG models with common noise. In \cite{carmona2014commonnoise}, Carmona et al. deal with a notion of \text{weak} solutions as opposed to the \text{strong} solutions constructed in this paper. They show, under great generality, the existence of a weak solution to MFG models with common noise. However, notice that while the assumptions in \cite{carmona2014commonnoise} are quite weak, they do not include many linear-quadratic models which verify our assumptions. In \cite{lacker2014translation}, Lacker and Webster show the existence of a strong solution, but only to a particular class of MFG models with common noise that satisfy a \textit{translation invariant} property. This class of models can be related to a certain MFG model without common noise via a simple transformation. Comparing \cite{lacker2014translation} to our work, they imposes a different type of restriction on the objectives and use different techniques to construct a strong solution.}

The paper is organized as follows. In Section \ref{model}, we introduce a model for an $N$-player stochastic differential game and formulate its limit, an MFG with common noise. We then review the Stochastic Maximum Principle along with some existence and uniqueness results of FBSDE in Section \ref{smp_fbsde}.  In Section \ref{mainresult}, we state and prove our main result, existence and uniqueness of a solution to MFG models with common noise. In Section \ref{lemma}, we provide the proofs for all the main lemmas used in Section \ref{mainresult}.

\section{Model}\label{model}
 In this section, we describe a stochastic differential game model with $N$ players, then formulate the limit problem as an MFG with common noise. \add{As mentioned above, we do not deal with the $N$-player games in this paper. The material in Section \ref{n_player} only serves as a motivation for the formulation of the MFG problem in Section \ref{mfg_formulation} and will not play a role in other parts of the paper.}  
 
 \subsection{N-player Stochastic Differential Games}\label{n_player} Let $T$ be a fixed terminal time, $W^{i} = (W^{i}_{t})_{0 \leq t \leq T}, i =1,2,\dots,N, \tW = (\tW_{t})_{0 \leq t \leq T}$ are one dimensional independent Brownian motions. Consider a stochastic dynamic game with $N$ players. Each player $i  \in \{1,2,\dots,N \}$ controls a state process $X^i_t$ in $\R$ given by 
$$ dX^{i}_{t} = \alpha^{i}_{t}dt + \sg dW^{i}_{t} + \tsg d\tW_t , \quad X^{i}_{0} = \xi^{i}_0$$
by selecting a control $\alpha^{i} = (\alpha^{i}_{t})_{0 \leq t \leq T}$ in $\Htwo$, the set of progressively measurable process $\beta = (\beta_t)_{0\leq t \leq T}$ such that
$$ \mb{E} \l[ \int_0^T \beta_t^2 dt \r] < \infty $$
where $\xi^i_0$ is an initial state of player $i$. We assume that $(\xi^i_0)_{1 \leq i \leq N}$ are independent identically distributed, independent of all Brownian motions, and satisfy $\mb{E}[(\xi_0^i)^2] < \infty$ for all $1 \leq i \leq N$. We will refer to $W^i$ as an \textit{individual noise} or \textit{idiosyncratic noise} and $\tW$ as a \textit{common noise.}

Let $\P(\R)$ denote the space of Borel probability measure on $\R$, $\Ptwo$ denote the subspace of $\P(\R)$ with finite second moment, i.e. a probability measure $m$ such that
$$ \int x^2 dm(x)  < \infty. $$
The space $\Ptwo$ is a complete separable metric space equipped with a Wasserstein metric defined as
\begin{equation}\label{wass}
 W_2(m_1,m_2) = \l( \inf_{\gamma \in \Gamma(m_1,m_2)} \int_{\R^2} |x-y|^2 \gamma(dx,dy) \r)^{\frac{1}{2}}
 \end{equation}
where $\Gamma(m_1,m_2)$ denotes the collection of all probability measures on $\R^2$ with marginals $m_1$ and $m_2$. 

Given the other players' strategy, player $i$ selects a control $\alpha^i \in \Htwo$ to minimize his/her expected total cost given by
$$ J^{i}(\alpha^{i} | (\alpha^{j})_{j\neq i}) := \mathbb{E}\l[ \int_{0}^{T} \frac{(\alpha_{t}^{i})^{2}}{2} dt + g(X^{i}_{T},m_{T})\r] $$
where $(\alpha^{j})_{j\neq i}$ denotes a strategy profile of other players excluding $i$, $g : \R \times \P(\R) \to \R$ is the terminal cost which we assume to be identical for all players, and $m_t$ is the empirical distribution of $(X^{i}_{t})_{1 \leq i \leq N}$, i.e. 
$$ m_t =  \frac{1}{N}\sum_{i=1}^{N}\delta_{X^{i}_{t}}(dx)  $$

Note that the strategies of the other players have an effect on the cost of player $i$ through $m_t$ and that is the main feature that makes this set up a game. We are seeking a type of equilibrium solution widely used in game theory settings called Nash equilibrium. 

\begin{definition} A set of strategies $(\alpha^{i})_{1 \leq i \leq N}$ is a \textit{Nash Equilibrium} if for every player $i$, $\alpha^{i}$ is optimal given the other players' strategies are $(\alpha^j)_{j \neq i}$. In other words, 
$$ J^i(\alpha^{i} |(\alpha^{j})_{j\neq i}) = \min_{\alpha \in \Htwo} J^i(\alpha | (\alpha^{j})_{j\neq i}),\quad \forall i \in \{1,2,\dots,N\}$$
\end{definition}

\del{Since the cost function of each player is identical, if the control problem for each player has a unique solution, then clearly the Nash equilibrium strategy must be symmetric. However, solving this problem}\add{Solving for a Nash equilibrium of an $N$-player game} is impractical when $N$ is large due to high dimensionality. So we formally take the limit as $N \to \infty$ and consider the limit\del{ing} problem, called a Mean Field Game (MFG), instead. Solving an MFG problem yields a control that can be used to approximate the exact Nash equilibrium for an $N$-player game. See \add{\cite{bensoussan2013,Carmona1,lacker2014general}} for a discussion and results on an approximate Nash equilibrium for $N$-player games. In this work, we will only show the well-posedness of an MFG problem in the presence of common noise. The approximation to $N$-player games will be treated in our future work.

 \subsection{MFG with Common Noise}\label{mfg_formulation} We now formulate a MFG problem in the presence of a common noise by formally taking a limit as $N \to \infty$. By considering the limiting problem and assuming that each player adopts the same strategy, one can represent the player distribution $m_t$ by a conditional law of a single \textit{representative} player given a common noise. In other words, we formulate the MFG with common noise as a stochastic control problem for a single (representative) agent with an equilibrium condition involving a conditional law of the state process given a common noise. 
 
 Let $T > 0$ be a fixed terminal time, $W=(W_t)_{0 \leq t \leq T}, \tW = (\tW_t)_{0 \leq t \leq T}$ be one dimensional independent Brownian motions defined on a complete filtered probability space $(\Omega,\F,\{ \F_{t} \}_{0\leq t \leq T},\mbb{P})$ satisfying the usual conditions. We assume that $\F_t$ is a natural filtration generated by $\xi_0, (W_s)_{0 \leq s \leq t}$, and $(\tW_s)_{0 \leq s \leq t}$, $\tF_t$ is generated by $(\tW_s)_{0 \leq s \leq T}$, both augmented by $\mb{P}$-null sets.  
 
 For a filtration $\G_t \subseteq \F_t$, we let $\Ltwo_{\G_t}$ denote the set of $\G_t$-measurable square integrable random variable and let $\H^2_\G([0,T];\R)$ denote the set of all $\G_t$-progressively-measurable process $\beta = (\beta_{t})_{0 \leq t \leq T}$ such that
$$ \mbb{E}\l[ \int_{0}^{T} \beta^{2}_{t} dt \r] < \infty$$
We define similarly the space $\H_\G^2([s,t];\R)$ for any $0 \leq s  <  t \leq T$ and we will often omit the subscript and write $\H^2([0,T];\R)$ for $\H^2_\F([0,T];\R)$. 
 
 Let $\xi_0 \in \Ltwo_{\F_0}$ be an initial state. For any control $\alpha \in \Htwo$, we denote by $X^\alpha = (X^\alpha_t)_{0 \leq t \leq T}$ the corresponding state process, i.e.
 $$ X^\alpha_t = \xi_0 + \int_0^t \alpha_t dt + \sigma W_t + \tsg \tW_t, $$
 and $m^\alpha_t = \Law(X^\alpha_t | \tF_t)$ for all $t \in [0,T]$ where $\Law(\cdot | \tF_t)$ denotes the conditional law given $\tF_t$.  It is easy to see that when $\alpha \in \Htwo$ and $\mb{E}[\xi_0^2] < \infty$, $m_t^\alpha$ has finite second moment a.s., i.e. $m^\alpha_t \in \Ptwo$ a.s. 
 
 \add{
 \begin{remark} An existence of a progressively-measurable version of $(\Law(X^\alpha_t | \tF_t))_{0 \leq t \leq T}$ is guaranteed by Lemma 1.1 in \cite{kurtz1988} for instance. In fact, in this case there exist a continuous version by the Kolmogorov continuity theorem. 
 \end{remark}}
 
 \textbf{Problem Definition.} A MFG with common noise is defined as follows: Find a control $\ha \in \Htwo$ such that, given $(m_t^{\ha})_{0 \leq t \leq T}$, $\ha$ is an optimal control for a stochastic control problem with state process
$$ dX_t = \alpha_t dt + \sigma dW_t + \tsg \tdW_t, \quad X_0 = \xi_0 $$
and cost
$$ J(\alpha | m^{\ha}) := \mb{E}\l[ \int_{0}^{T} \frac{\alpha_{t}^{2}}{2} dt + g(X_{T},m^{\ha}_{T})\r] $$
In other words, it satisfies
$$ J(\ha| m^{\ha}) \leq J(\alpha | m^{\ha}), \quad \forall \alpha \in \Htwo$$

We will often refer to the MFG problem described above as \textit{$\tsg$-MFG} to emphasize the existence of a common noise and call $\ha$ a \textit{solution to $\tsg$-MFG with initial $\xi_0$}. Note that $0$-MFG is simply the original MFG with independent Brownian motions. 

We would like to emphasize that in the control problem above, $m^{\ha}_t$ is exogenous and is not affected by a player's control. Thus, MFG is a standard control problem with an additional equilibrium condition. A type of problem where a player's control can affect the law is referred to as \textit{Mean Field Type Control Problem}. See \cite{bensoussan2013,carmona2013,carmona2013control} for a treatment on this different model and some discussion of the distinctions between the two problems. 
 
 Alternatively, we can view the $\tsg$-MFG problem as a fixed point problem as follows: Given a strategy $\alpha \in \Htwo$, then $m^{\alpha}_t$ is determined as defined above. With $m^{\alpha}_{t}$ given as a random flow of probability measures, we can solve a control problem for a representative player. This step yields a new optimal control $\tilde{\alpha}$. The following diagram summarizes the process
\begin{equation}\label{diagram}
	\ha = (\ha_{t})_{ 0 \leq t \leq T} \,\to\, m^{\ha} = (m^{\ha}_{t})_{0 \leq t \leq T} \,\to\, \tilde{\alpha} = (\tilde{\alpha}_{t})_{0 \leq t \leq T} 
\end{equation}
By the definition of the $\tsg$-MFG problem, $\ha$ is a $\tsg$-MFG solution if and only if it is a fixed point of this map.

Our main result in this paper is to show existence and uniqueness of a solution to $\tsg$-MFG, i.e. a fixed point of diagram (\ref{diagram}). We will define this map formally through FBSDE after discussing the Stochastic Maximum Principle in the next section.

\section{Stochastic Maximum Principle}\label{smp_fbsde}
Stochastic Maximum Principle (SMP for short) or Pontryagin Maximimum Principle is an approach to a control problem which studies optimality conditions satisfied by an optimal control. It gives sufficient and necessary conditions for the existence of an optimal control in terms of solvability of a Backward Stochastic Differential Equation (BSDE for short) of an \textit{adjoint} process. In this section, we apply SMP to the control problem in the MFG model. For more details about SMP, we refer to \cite{pham2009} or \cite{young1999}.

To apply sufficient conditions for the SMP, the convexity assumption on the cost functions is needed. Throughout the rest of the paper, we assume that $g: \R \times \Ptwo \to \R$ satisfies

\begin{assump}\label{a1} (Lipschitz in $x$) For each $m \in \Ptwo$, $g_x(x,m)$ exists and is Lipschitz continuous in $x$ i.e. there exist a constant $C_g$ such that
	$$ |g_x(x,m)-g_x(x',m)|  \leq C_g |x-x'| $$
	for any $x,x' \in \R$
\end{assump}
\begin{assump}\label{a2} (convexity) For any $x,x' \in \R, m \in \Ptwo$, 
	\begin{equation}\label{convexity}	
		(g_x(x,m)-g_x(x',m))(x-x') \geq 0
	\end{equation}
\end{assump}

We are ready to state the SMP for our control problem. Suppose we are given a random probability measure $m_T(\omega)$, then we have a standard control problem with a random terminal cost function given by $g(\cdot,m_T(\omega))$. The SMP for our model then reads

\begin{theorem}\label{SMP}
Suppose that there exist an adapted solution $(\hX_t,\hY_{t},\hZ_{t},\htZ_t)_{0 \leq t \leq T}$ to the FBSDE
\begin{equation}
	\label{fbsde}
	\begin{gathered}
 dX_{t} = -Y_{t} dt + \sigma dW_{t} + \tsg \tdW_{t} \\
 dY_{t}  = Z_{t} dW_{t} + \tZ_{t}\tdW_{t} \\
 X_{0} = \xi_{0}, \quad Y_{T} = g_{x}(X_{T},m_{T}) 
	\end{gathered}
\end{equation}
such that
$$ \mb{E}\l[ \sup_{0 \leq t \leq T} (\hX_t^2 + \hY_t^2) + \int_0^T \hZ^2_t + \htZ_t^2 dt \r] < \infty, $$
then $\ha_t = -\hY_t$ gives an optimal control to the control problem given $m_T$. Furthermore, for any $\beta \in \Htwo$, the following estimate holds
\begin{equation}\label{estJ}
   J(\ha | m) + \frac{1}{2}|\ha-\beta|^2 \leq J(\beta | m) 
   \end{equation}
Particularly, $\ha_t$ is the unique optimal control.
\end{theorem}

\begin{proof} The proof is standard and we refer to Theorem 6.4.6 in \cite{pham2009}. The estimate \eqref{estJ} requires strict convexity in $\alpha$ of the running cost function. The proof can be found in Theorem 2.2. in \cite{Carmona1}.  
\end{proof}

Having stated the SMP, we now show that the FBSDE above is uniquely solvable, thereby proving that the control problem for a fixed $m_T$ is uniquely solvable. We state the result slightly more generally by allowing for a random terminal function and an arbitrary initial and terminal time as it will be applied again in a subsequent section. This result is an immediate consequence of Theorem 2.3 in \cite{peng1999} which gives the existence and uniqueness of a solution to a monotone FBSDE. 

\begin{theorem}\label{EUexample} Let $0 \leq s \leq \tau \leq T$ and $\xi \in \Ltwo_{\F_\tau}$. Suppose $v : \R \times \Omega \to \R$ is a $\F_\tau$-measurable Lipschitz continuous function satisfying a monotonicity condition
$$ (v(x,\omega)-v(x',\omega))(x-x') \geq 0 $$
for all $x,x' \in \R$ and $\omega \in \Omega$. Then there exist a unique adapted solution \del{$(X_{t},Y_{t},Z_{t},\tZ_{t})_{0 \leq t \leq T}$} \add{$(X_{t},Y_{t},Z_{t},\tZ_{t})_{s \leq t \leq \tau}$} to  
\begin{equation}
	\begin{gathered}
 dX_{t} = -Y_{t} dt + \sigma dW_{t} + \tsg \tdW_{t} \\
 dY_{t}  = Z_{t} dW_{t} + \tZ_{t}\tdW_{t} \\
 X_{s} = \xi, \quad Y_{\tau} = v(X_{\tau}) 
	\end{gathered}
\end{equation}
satisfying the following estimate
$$ \mb{E}\l[ \sup_{s \leq t \leq \tau} [ |X_t|^2 + |Y_t|^2 ] + \int_s^\tau [ |Z_t|^2 + |\tZ_t|^2 ] dt  \r] \leq C( \mb{E}[\xi^2] + \mb{E}[v^2(0)] + (\tau-s)(\sigma^2 + \tsg^2)) $$
for some constant $C$ depending on $T$ and the Lipschitz constant of $v$.

\end{theorem}

As a result of Theorem \ref{SMP} and Theorem \ref{EUexample}, we have the following corollary which shows that, for any given $m_T$, the control problem for a representative player is uniquely solvable, i.e. there exist a unique optimal control.

\begin{corollary}\label{SMPexample} Suppose \ref{a1}-\ref{a2} hold, then for any given $m_T$, the control problem described above has a unique optimal control given by $\ha_t = -Y_t$, where  $(X_{t},Y_{t},Z_{t},\tZ_{t})_{0 \leq t \leq T}$ is the solution to the FBSDE (\ref{fbsde}).
\end{corollary}

\begin{remark}
 In the context of Dynamic Programming Principle (DPP), which involves an HJB equation for the value function, the adjoint equation of the Stochastic Maximum Principle, under some regularity condition, is simply the backward SDE of the gradient of the value function. For more details about the connection between SMP and DPP, we refer to \del{the Appendix}\add{Theorem 6.4.7 in \cite{pham2009} for instance}.  
 \end{remark}

\section{Existence and Uniqueness of $\tsg$-MFG solution}\label{mainresult} In this section, we state and prove our main result which establishes an existence and uniqueness of a solution to Mean Field Games with a common noise.

In \cite{Carmona1}, Carmona and Delarue showed existence of $0$-MFG solution using SMP approach. They established the result under a convexity assumption and what they called a \textit{weak mean-reverting} assumption on the cost function. The latter states that for all $x \in \R$ and a constant $C > 0$.
\begin{equation}\label{weakMR}
 xg_x(0,\delta_x) \geq -C(1+|x|)
 \end{equation}
 where $\delta_x$ denote the Dirac measure at $x$. They first proved the result under a bounded derivative assumption on the  function, then relaxed this assumption using an approximation argument. 
 
 In \cite{cardaliaguet2010}, using PDE method, Lasry and Lions showed the same result under similar assumptions plus boundedness condition on the cost function. They gave existence of a classical solution to a coupled HJB-FP equation. The uniqueness was shown in both works \add{\cite{cardaliaguet2010,Carmona1}} under the following \textit{monotonicity} condition (in $m$) on the cost function, i.e
$$ \int (g(x,m_1)-g(x,m_2))d(m_1-m_2)(x) \geq 0 $$
for any $m_1,m_2 \in \Ptwo$. This condition can be expressed in terms of random variables as follows: For any \del{$\xi,\xi' \in \Ltwo$ (over an arbitrary probability space).} \add{square-integrable random variables $\xi$ and $\xi'$ defined on a common probability space,}
\begin{equation}\label{monLion}
 	 \mb{E}\l[ g(\xi',\Law(\xi')) +g(\xi,\Law(\xi)) - g(\xi,\Law(\xi')) - g(\xi',\Law(\xi)) \r] \geq 0
\end{equation}
where $\Law(\cdot)$ denote the law of a random variable. 

In the existence proof of both the Lasry and Lions PDE approach and the Carmona and Delarue probabilistic approach, they applied Schauder's theorem to establish an existence of a fixed point. However, this strategy cannot be extended simply in the presence of common noise as we no longer have a deterministic flow of probability measures given by the law of the optimal state process. Instead, we have to deal with a \textit{random} flow of probability measures from a \textit{conditional} law given a common noise. Working with this larger space, we cannot establish compactness which is necessary to apply the Schauder fixed point theorem in the same way. 

\add{In parallel to our work, Carmona et al. \cite{carmona2014commonnoise} overcome this issue by using discretization of the common noise path. Working with this finite-dimensional approximation, they were able to prove the existence of a solution to the discretized MFG and obtain the existence of the original MFG by refining the approximation and passing the limit. This construction gives the \textit{weak} limit and leads to a notion of a  \textit{weak MFG} solution.} 

\del{As a result, we will use} \add{In this work, however, we adopt an alternative approach}, namely the Banach fixed point theorem. In the same way as in the proof of wellposedness of FBSDE, the Banach fixed point theorem can be used to establish existence of a solution when the time duration $T$ is sufficiently small. This method can usually be applied in that case because the solution estimate depends on $T$,  and when $T$ is sufficiently small, we can get a contraction map easily. See \cite{Antonelli,Ma1999} for proofs of existence and uniqueness of a solution to FBSDE for a small time duration. However, the small time restriction is not a desirable assumption for obvious reasons, so we wish to extend the solution to arbitrary time duration. To do so, we need an extra condition on $g$ to be able to control the Lipschitz constant of a new terminal condition as we move backwards in time. So in addition to $\ref{a1},\ref{a2}$, we will need the following assumptions on $g$

\begin{assump}\label{a3}(Lipschitz in $m$) $g_x$ is Lipschitz continuous in $m$ uniformly in $x$, i.e. there exist a constant $C_g$ such that
	$$ |g_x(x,m)-g_x(x,m')|  \leq C_g W_2(m,m') $$
	for all $x \in \R, m,m' \in \Ptwo$, where $W_2(m,m')$ is the second order Wasserstein metric defined by (\ref{wass}). This is equivalent to the following; for any $x \in \R$, \del{$\xi,\xi'  \in \Ltwo$} \add{and any square-integrable random variables $\xi$ and $\xi'$ defined on a common probability space,}
	$$ |g_x(x,\Law(\xi))-g_x(x,\Law(\xi'))|  \leq C_g (\mb{E}|\xi-\xi'|^2)^{\frac{1}{2}} $$
\end{assump}
\begin{assump}\label{a4}(weak monotonicity in $m$) For any $m,m' \in \Ptwo$ and any $\gamma \in \P_2(\R^2)$ with marginals $m,m'$ respectively,
	$$ \int_{\R^2} \l[ (g_x(x,m)-g_x(y,m'))(x-y) \r] \gamma(dx,dy) \geq 0 $$
	Equivalently, for any \del{$\xi,\xi' \in \Ltwo$} \add{square-integrable random variables $\xi$ and $\xi'$ defined on a common probability space,} 
	\begin{equation}\label{monotone}
		\mb{E}[ g_x(\xi,\Law(\xi)) - g_x(\xi',\Law(\xi'))(\xi-\xi')] \geq 0
	\end{equation}
\end{assump}

\begin{remark}\label{remark_example} An interesting example of $g$ which satisfies \ref{a1}-\ref{a4} is a general quadratic cost where $g$ is of the form
\begin{equation}\label{LQexample}
 g(x,m) = Ax^2 + x \int \psi(y)dm(y) + F(m)
 \end{equation}
 where $\psi:\R \to \R$ and $F: \Ptwo \to \R$ are Lipschitz continuous functions with $\text{Lip}(\psi) \leq 2A $. This form of $g$ includes the following examples  
 \begin{gather}
 	g(x,m) = \l( x - \int y dm(y) \r)^2 \label{e1}\\
 	g(x,m) =  \int (x-y)^2 dm(y) \notag
\end{gather}
 which occur frequently in applications (see \cite{carmona2013mean,gueant2011} for instance). More generally, it includes the cost function in the Linear-Quadratic Mean Field Games \cite{bensoussan2014linear} where $g$ takes the form
 $$ g(x,m) = qx^2 + \bar{q}\l(x-s\int y m(y) \r)^2$$
 where $q,\bar{q},s$ are positive constants satisfying $q+\bar{q} \geq s\bar{q}$. 
\end{remark}

Our main result establishes, under \ref{a1}-\ref{a4}, the existence and uniqueness of a $\tsg$-MFG solution. In the existence proof, our assumptions are similar to those in \cite{Carmona1} except that we replace the weak mean-reverting condition (\ref{weakMR}) by the weak monotonicity condition (\ref{monotone}). When $\xi = x, \xi' = 0$ is deterministic, (\ref{monotone}) reads
$$ (g_x(x,\delta_x) - g_x(0,\delta_{0}))x \geq 0, \qquad \forall x \in \R$$
where $\delta_x$ denote the Dirac measure at $x$. This implies (\ref{weakMR}), so one can view (\ref{monotone}) as a stronger version of (\ref{weakMR}). For the uniqueness result, when $g$ is convex in $x$, the weak monotonicity condition (\ref{monotone}) is indeed a weaker version of the monotonicity condition (\ref{monLion}) which was used in \cite{cardaliaguet2010,Carmona1} for the uniqueness proof of a $0$-MFG solution. The lemma below shows this.

\begin{lemma}\label{weakerMon} Let $h: \R \times \Ptwo$ be a continuously differentiable function such that $h$ is convex and satisfies the monotonicity condition (\ref{monLion}) stated above, then $h$ satisfies the weak monotonicity condition (\ref{monotone}).
\end{lemma}

\begin{proof} Suppose (\ref{monLion}) holds. Let $\xi,\xi' \in \Ltwo$, then by convexity in $x$ of $h$, we get
$$ h(\xi',\Law(\xi')) - h(\xi,\Law(\xi')) \leq h_x(\xi',\Law(\xi'))(\xi'-\xi) $$
and
$$ h(\xi,\Law(\xi)) - h(\xi',\Law(\xi)) \leq -h_x(\xi,\Law(\xi))(\xi'-\xi) $$
 Summing up, taking expectation, and using (\ref{monLion}) yields (\ref{monotone}).
 
\end{proof}


The converse of the lemma above does not hold.  Example (\ref{e1}) gives a cost function that is convex in $x$, satisfies (\ref{monotone}), but does not satisfy (\ref{monLion}). Thus, one see that our result gives a more general uniqueness result for $0$-MFG when $g$ is convex in $x$.

\subsection{Uniqueness}\label{uniqueness} We begin by showing uniqueness of a $\tsg$-MFG solution	

\begin{theorem}[Uniqueness] \label{main2} Suppose $\ref{a1}-\ref{a4}$ holds, then the $\tsg$-MFG with initial condition $\xi_0 \in \Ltwo_{\F_0}$ has at most one solution.
\end{theorem}

\begin{proof}  Let $\ha^1,\ha^2 \in  \Htwo $ be solutions to $\tsg$-MFG with initial $\xi_0 \in \Ltwo_{\F_0}$, and let $m_t^i = \Law(X^{\ha^i}_t| \tF_t)$ for  $i=1,2$. \add{Recall that $(\tF_t)_{0 \leq t \leq T}$ is the filtration generated by the common noise $(\tW_t)_{0 \leq t \leq T}$.} By Theorem \ref{EUexample} and Corollary \ref{SMPexample}, we have that for $i \in \{1,2\}$,  $\ha^i= -Y^i$ where $(X^i_t,Y^i_t,Z^i_t,\tZ^i_t)_{0\leq t \leq T}$ is the solution to the FBSDE
\begin{gather*}
	 dX_t = -Y_tdt + \sigma dW_t + \tsg \tdW_t \\
	 dY_t =  Z_tdW_t + \tZ_t \tdW_t \\
	 X_0 = \xi_0, Y_T = g_x(X_T,m^i_T)
\end{gather*}
From the fact that $\ha^i$ is the solution to $\tsg$-MFG, we get that $m^i_t$ is precisely the law of $X^i_T$ conditional on $\tF_t$. That is, $(X^i_t,Y^i_t,Z^i_t,\tZ^i_t)_{0\leq t \leq T}$ satisfies, for $i=1,2$,
\begin{gather*}
	 dX^i_t = -Y^i_tdt + \sigma dW_t + \tsg \tdW_t \\
	 dY^i_t =  Z^i_tdW_t + \tZ^i_t \tdW_t \\
	 X^i_0 = \xi_0, Y^i_T = g_x(X^i_T, \Law(X^i_T | \tF_T))
\end{gather*}
Let $\Delta X_t = X^1_t-X^2_t$ and $\Delta Y_t = Y^1_t-Y^2_t$. Using \ito's lemma, we get
$$ \mb{E}[\Delta X_T\Delta Y_T] - \Delta X_0 \Delta Y_0 = - \mb{E}\l[ \int_0^T \Delta Y_t^2 dt \r] $$
Since $\Delta X_0 = 0$ and from (\ref{monotone}), it follows that
\begin{align*}
\mb{E}[\Delta X_T \Delta Y_T] &= \mb{E}[ (X_T^1-X_T^2)(g_x(X_T^1,\Law(X_T^1|\tF_T))-g_x(X_T^2,\Law(X_T^2|\tF_T))] \\
&= \mb{E} \l[ \mb{E}\l[ (X_T^1-X_T^2)(g_x(X_T^1,\Law(X_T^1|\tF_T))-g_x(X_T^2,\Law(X_T^2|\tF_T))\Big| \tF_T \r]\r]  \geq 0
\end{align*}
Thus,
$$   \mb{E}\l[ \int_0^T \Delta Y_t^2 dt \r]   = -  \mb{E}[\Delta X_T\Delta Y_T]  \leq 0 $$
so $\ha^1 = \ha^2$ in $\Htwo$.
 
\end{proof}

\subsection{Existence} Next, we state and prove the existence result for $\tsg$-MFG.

\begin{theorem}[Existence]\label{main1} Under \ref{a1}-\ref{a4}, there exist a solution to $\tsg$-MFG with initial $\xi_0 \in \Ltwo_{\F_0}$.
\end{theorem} 

Before we delve into the  proof, let us first define a following map whose fixed point gives us a $\tsg$-MFG solution. This map is simply the rigorous definition of the diagram (\ref{diagram}) in Section \ref{model}  .

Let $0 \leq s < \tau \leq T$ and $v: \R \times \Ltwo_{\F_\tau} \times \Omega \to \R$ satisties, for $\mb{P}$-a.s.,
\begin{gather}
	 (v(x,\xi, \omega)-v(x',\xi,\omega))(x-x') \geq 0  \label{v1} \\
	 | v(x, \xi,\omega )-v(x',\xi',\omega) |^2  \leq C_v \l[ |x-x'|^2 + \mb{E}\l[(\xi-\xi')^2 \Big| \tF_\tau \r](\omega) \r]  \label{v2}
\end{gather}
for all $x,x' \in \R, \xi,\xi' \in \Ltwo_{\F_\tau}$.

We define a map $\Phi^{s,\tau,\eta,v} : \H^{2}([s,\tau];\R) \to \H^{2}([s,\tau];\R)$ as follows; given $\ha \in  \H^{2}([s,\tau];\R)$, we define $(\hX_t)_{s \leq t \leq \tau}$ to be the state process corresponding to $\ha$ with initial $\eta \in \Ltwo_{\F_s}$, i.e. for $t \in [s,\tau]$, 
\begin{equation}\label{stateProcess}
 \hX_{t} =  \eta + \int_s^t \ha_{t} dt + \sigma (W_t - W_s) + \tsg (\tW_t - \tW_s) 
 \end{equation}
We then solve the FBSDE
\begin{equation}
	\label{fbsde_phi}
	 \begin{gathered}
dX_{t} = -Y_{t}dt + \sigma dW_{t} + \tsg \tdW_{t}   \\
dY_{t} = Z_{t}dW_{t} +  \tZ_{t}\tdW_{t}\\
X_{s}=\eta,\quad Y_{\tau} = v(X_{\tau},\hX_{\tau}) 
	\end{gathered}
\end{equation}
and set
$$ \Phi^{s,\tau,\eta,v}(\ha) := -Y = (-Y_t)_{s \leq t \leq \tau}$$
By (\ref{v1}),(\ref{v2}), and Theorem \ref{EUexample}, FBSDE (\ref{fbsde_phi}) is uniquely solvable so the map $\Phi^{s,\tau,\eta,v}$ is indeed well-defined. Furthermore, the fixed point of $\Phi^{0,T,\xi_0,g_x}$ clearly gives $\tsg$-MFG solution with initial $\xi_0$. We are now ready to begin the proof of Theorem \ref{main1}

\begin{proof} We start by considering the map $\Phi^{t,T,\eta,g_x}$, which is well-defined by $\ref{a1}, \ref{a2}$, and Theorem \ref{EUexample}. The lemma below gives a solution over a small time duration.

\begin{lemma}\label{phi3} There exist $\gamma > 0$ depending only on $C_g$ such that for any non-negative $t \in [T-\gamma,T), \Phi^{t,T,\eta,g_x}$ satisfies the following: for any $\eta \in \Ltwo_{\F_t}$, there exist $\ha^{t,T,\eta} \in \H^{2}([t,T];\R)$ such that
$$ \Phi^{t,T,\eta,g_x}(\ha^{t,T,\eta}) = \ha^{t,T,\eta} $$
\end{lemma}

This lemma is a special case of Lemma \ref{phi2} which is stated below. The key point of this lemma is the fact that the small time duration $\gamma$ is independent of $\eta$. It requires us to prove that $\Phi^{t,T,\eta,g_x}$ is a contraction map uniformly in $\eta$ when $T-t$ is sufficiently small.

Let $\Gamma$ be the set of $t \in [0,T]$ such that the following statement holds; for any $\eta \in \Ltwo_{\F_t}$, there exist $\ha^{t,T,\eta} \in \H^{2}([t,T];\R)$ satisfying
$$  \Phi^{t,T,\eta,g_x}(\ha^{t,T,\eta}) = \ha^{t,T,\eta} $$
By Lemma \ref{phi3} above, $T-\gamma \in \Gamma$. If $0 \in \Gamma$, then we have completed the proof. Suppose not, let $t_0 = \inf \Gamma$ ($t_0$ can still be zero),  $\gamma_0 > 0$ be sufficiently small so that $t_0+\gamma_0 < T-\frac{\gamma}{2} $, and let $\tau \in [t_0,t_0+\gamma_0) \cap \Gamma$. From the fact that $\tau \in \Gamma$, for any initial $\eta \in \Ltwo_{\F_\tau}$, there exist a fixed point $(\ha^{\tau,\eta}_t)_{\tau \leq t \leq T} \in \H^2([\tau,T];\R)$ of the map $\Phi^{\tau,T,\eta,g_x}$. Let $(\hX^{\tau,\eta}_t)_{\tau \leq t \leq T}$ denote the corresponding state process, i.e. for $t \in [\tau,T]$,
\begin{equation}\label{refProcess}
	 \hX_t^{\tau,\eta} = \eta + \int_\tau^t\ha^{\tau,\eta}_t dt + \sigma(W_t - W_\tau) + \tsg (\tW_t - \tW_\tau)
\end{equation}
Consider the following FBSDE
\begin{gather}
dX_{t} = -Y_{t}dt + \sigma dW_{t} + \tsg \tdW_{t}   \nonumber \\
dY_{t} = Z_{t}dW_{t} +  \tZ_{t}\tdW_{t} \label{fbsde1}\\
X_{\tau}=x,\quad Y_{T} = g_{x}(X_{T},\Law(\hX^{\tau,\eta}_T|\tF_T)) \nonumber
\end{gather}
By Theorem \ref{EUexample}, the FBSDE (\ref{fbsde1}) above is uniquely solvable and we denote the solution by $(X_t^{\tau,x,\eta},Y_t^{\tau,x,\eta},Z_t^{\tau,x,\eta},\tZ^{\tau,x,\eta}_t)_{\tau \leq t \leq T}$. Let $u: \R \times \Ltwo_{\F_\tau} \times \Omega \to \R$ be defined as $u(x,\eta,\omega) = Y_\tau^{\tau,x,\eta}(\omega)$, then we have the following estimates whose proofs are given in the next section.

\begin{lemma}\label{decoupling} $u$ defined above satisfies, \del{for} $\mb{P}$-a.s.,
\begin{gather}
	 (u(x,\xi, \omega)-u(x',\xi,\omega))(x-x') \geq 0  \label{bddux} \\
	 | u(x, \xi,\omega )-u(x',\xi',\omega) |^2  \leq C_u \l[ |x-x'|^2 + \mb{E}\l[(\xi-\xi')^2 \Big| \tF_\tau \r](\omega) \r]  \label{bddum}
\end{gather}
for all $x,x' \in \R, \xi,\xi' \in \Ltwo_{\F_\tau}$ with the constant $C_u$ depending only on $T, C_g, \gamma$.
\end{lemma}

Next, we attempt to extend the solution further. Let $s \in [0,\tau)$ to be determined, $\eta \in \Ltwo_{\F_s}$ and consider the map $\Phi^{s,\tau,\eta,u}$ as defined above. We have the following lemma which enables us to extend the solution beyond $t_0$ creating a contradiction.

\begin{lemma}\label{phi2} Suppose $u: \R \times \Ltwo_{\F_\tau} \times \Omega \to \R$ satisfies (\ref{bddux}) and (\ref{bddum}), then there exist $\gamma' > 0$ depending only on $C_u$ such that for any non-negative $s \in [\tau-\gamma',\tau)$, the following holds: For any $\eta \in \Ltwo_{\F_s}$, there exist $\ha^{s,\tau,\eta} \in \H^{2}([s,\tau];\R)$ satisfying
$$ \Phi^{s,\tau,\eta,u}(\ha^{s,\tau,\eta}) = \ha^{s,\tau,\eta} $$
\end{lemma}

Note that this lemma implies Lemma \ref{phi3} above by setting $u(x,\xi) = g_x(x,\Law(\xi|\tF_T))$. Assumption \ref{a1}-\ref{a4} implies that $g_x$ satisfies (\ref{bddux}) and (\ref{bddum}). The proof of this lemma is given in the next section. Next, we let $\gamma'$ be a constant from Lemma \ref{phi2} and $s$ be a non-negative element in $[\tau-\gamma',\tau)$. We construct a following control $\ha^{s,T,\eta}$ by letting
$$ \ha^{s,T,\eta} = \begin{cases} \ha^{s,\tau,\eta}, &\text{if }s\leq t < \tau \\  \ha^{\tau,T,\hX_\tau}, &\text{if }\tau \leq t \leq T \end{cases} $$
By definition of $u$, it follows that $ \ha^{s,T,\eta}$ satisfies
$$ \Phi^{s,T,\eta,g_x}(\ha^{s,T,\eta}) = \ha^{s,T,\eta} $$
We are left to check that $\ha^{s,T,\eta}$ is in $\H^{2}([s,T];\R)$. To do so, we need the following estimate
\begin{lemma}\label{alpha_estimate}
Suppose $(\hX_t,\hY_t,\hZ_t,\htZ_t)_{\tau \leq t \leq T} \in (\H^2([\tau,T];\R))^4$ satisfies the FBSDE
\begin{gather*}
	d\hX_{t} = -\hY_{t}dt + \sg dW_{t} +\tsg \tdW_t					 \\
	d\hY_{t} = \hZ_{t}dW_{t} + \htZ_{t}\tdW_{t}			 \\
	\hX_{\tau} = \xi, \quad \hY_{T} = g_x(\hX_{T},\Law(\hX_{T}| \tF_T))  
\end{gather*}
where $\xi \in \Ltwo_{\F_\tau}$, then the following estimate holds:
\begin{equation}\label{fbsde_full_est}
 \mb{E}\l[ \sup_{\tau \leq t \leq T} [ |X_t|^2 + |Y_t|^2 ] \r] \leq C_{T,C_g}(\mb{E}[\xi^2] +g^2_x(0,\delta_0)+\sigma^2 + \tsg^2) 
 \end{equation}
\end{lemma}

Let $\hX = (\hX_t)_{s \leq t \leq T}$ denote the state process corresponding to the control $\ha^{s,T,\eta}$. Then by using estimate (\ref{fbsde_full_est}), we get
\begin{align*}
\mb{E} \l[ \int_s^T( \ha^{s,T,\eta})^2 dt \r]  &=  \mb{E} \l[ \int_s^\tau( \ha^{s,\tau,\eta})^2 dt  \r]+ \mb{E} \l[\int_\tau^T (\ha^{\tau,T,\hX_\tau})^2dt  \r]  \\
& \leq   \mb{E} \l[ \int_s^\tau( \ha^{s,\tau,\eta})^2 dt  \r] + C_{T,C_g} \l[ \mb{E}[\hX^2_\tau]  + g^2_x(0,\delta_0) + \sigma^2 + \tsg^2 \r] \\
&\leq  \tilde{C}_{T,C_g} \l( \mb{E} \l[ \int_s^\tau(\ha^{s,\tau,\eta})^2dt \r] + \mb{E}[\eta^2] + g^2_x(0,\delta_0) + \sigma^2 + \tsg^2 \r) < \infty
\end{align*}
which implies $\ha^{s,T,\eta} \in \H^{2}([s,T];\R)$.  The proof of Lemma \ref{alpha_estimate} is presented in the next section. 

We have shown that $s \in \Gamma$. From Lemma \ref{phi2}, we know that $\gamma'$ depends only on $C_u$ which is independent of $\tau$. Therefore, we can select $\gamma_0$ sufficiently small so that $\gamma_0 <\gamma' $, so that $\tau - \gamma' \leq t_0+\gamma_0-\gamma' < t_0$. Thus, we can select $s$ to be strictly less than $t_0$ or $s=0$ creating a contradiction as we assumed $0 \notin \Gamma$ and $t_0 = \inf \Gamma$. Therefore, $0 \in \Gamma$ and we have completed the proof.

\end{proof}


\section{Main Lemma}\label{lemma}

In this section, we prove all the lemmas that was used in Theorem \ref{main1}. The first one is a priori estimate of the so-called ``decoupling field" while the second lemma is a construction of $\tsg$-MFG solution using Banach fixed point theorem over a small time duration. The third lemma give an FBSDE estimate needed to ensure that the extended control remains admissible.

\subsection{Proof of Lemma \ref{decoupling}}
  
Let $x_{1},x_{2} \in \R$ and $(X^{i}_{t},Y^{i}_{t},Z^{i}_{t},\tZ^{i}_{t})_{\tau \leq t \leq T}$ denote the solution to the FBSDE (\ref{fbsde1}) with initial condition $X^{i}_{\tau} = x_{i}, i=1,2$. Let $\Delta Y_{t} = Y^{1}_{t}-Y^{2}_{t}$ and $\Delta X_{t} = X^{1}_{t}-X^{2}_{t}$, then by applying \ito's lemma to $\Delta Y_t \Delta X_t$ from $\tau$ to $T$ and taking expectation conditional on $\F_\tau$ (denote by $\mb{E}_\tau[\cdot])$, we get 
$$ \mb{E}_\tau[\Delta Y_T \Delta X_T] - \Delta Y_\tau \Delta X_\tau = - \mb{E}_{\tau}\l[ \int_{\tau}^{T}\Delta Y_{t}^{2} dt \r] $$
By convexity assumption \ref{a2}, it follows that
$$ \Delta Y_T \Delta X_T = \l(g_{x}(X^{1}_{T},\Law(\hX^{\tau,\eta}_T| \tF_T))-g_{x}(X^{2}_{T},\Law(\hX^{\tau,\eta}_T| \tF_T))\r)\Delta X_{t} \geq 0$$
Thus,
 $$ \Delta Y_{\tau} \Delta X_{\tau} = \mb{E}_{\tau}\l[\Delta Y_T \Delta X_T \r] +  \mb{E}_{\tau}\l[ \int_{\tau}^{T}\Delta Y_{t}^{2} dt \r]  \geq 0  $$
 That is, 
 $$ (u(x_{1},\eta)-u(x_{2},\eta))(x_{1}-x_{2}) \geq 0$$



Now we proceed to show (\ref{bddum}). Let $\eta_1,\eta_2 \in \Ltwo_{\F_\tau}$ and we consider $u(x_1,\eta_1)$ and $u(x_2,\eta_2)$. For $i=1,2$, let  $(X^i_{t},Y^i_{t},Z^i_{t},\tZ^i_{t})_{\tau \leq t \leq T}$ denote the solution to the corresponding FBSDE (\ref{fbsde1}) from the definition of $u(x_i,\eta_i)$. In other words, they satisfy
\begin{gather}
	dX^i_{t} = -Y^i_{t}dt + \sg dW_{t} +\tsg \tdW_t	\nonumber				 \\
	dY^i_{t} = Z^i_{t}dW_{t} + \tZ^i_{t}\tdW_{t}	\label{fbsdeInd}		 \\
	X^i_{\tau} = x_i, Y^i_{T} = g_{x}(X^i_{T},\Law(\hX^{\tau,\eta^i}_T|\tF_T))  \nonumber
\end{gather}
where $(\hX_t^{\tau,\eta^i})_{\tau \leq t \leq T}$ is given by
$$ \hX_t^{\tau,\eta} = \eta + \int_\tau^t\ha^{\tau,\eta}_t dt + \sigma(W_t - W_\tau) + \tsg (\tW_t - \tW_\tau), $$
and $(\ha^{\tau,\eta_i}_t)_{\tau \leq t \leq T}$ is the $\tsg$-MFG solution over $[\tau,T]$ with initial $\eta_i$, i.e. the unique fixed point of the map $\Phi^{\tau,T,\eta^i,g_x}$. For notational convenience, we write $\hX^{\tau,\eta^i}_t$ as $\hX^i_t$. We begin by getting a bound on $\hX^1_t-\hX^2_t$ before we go back to bound $u(x_1,\eta_1) - u(x_2,\eta_2)$. From the definition of $\Phi^{\tau,T,\eta^i,g_x}$ and the fact that $(\ha^{\tau,\eta_i}_t)_{\tau \leq t \leq T}$ is the unique fixed point of this map, it follows that $(\hX^i_t)_{\tau \leq t \leq T}$ is the forward process of the unique solution of FBSDE
\begin{gather*}
	d\hX^i_{t} = -\hY^i_{t}dt + \sg dW_{t} +\tsg \tdW_t					 \\
	d\hY^i_{t} = \hZ^i_{t}dW_{t} + \htZ^i_{t}\tdW_{t}			 \\
	\hX^i_{\tau} = \eta^i, \hY^i_{T} = g_{x}(\hX^i_{T},\Law(\hX^i_T|\tF_T))  
\end{gather*}
Let $\Delta \hX_{t} = \hX^1_{t}-\hX^2_{t}$ and define $\Delta \hY_{t}, \Delta \hZ_{t}, \Delta \hat{\tZ}_{t}, \Delta X_t, \Delta Y_t, \Delta Z_t, \Delta \tZ_t, \Delta x, \Delta \eta $ similarly. Then note that $(\Delta \hX_{t},\Delta \hY_{t}, \Delta \hZ_t, \Delta \htZ_t)_{\tau \leq t \leq T}$ satisfies
\begin{gather}
 d\Delta \hX_{t} = - \Delta \hY_{t} dt \nonumber \\
 d\Delta \hY_{t} = \Delta \hZ_{t} dW_{t} + \Delta \htZ_{t} \tdW_{t} \label{fbsdeDelta} \\
 \Delta \hX_{\tau} = \Delta \eta, \quad \Delta \hY_{T} = g_{x}(\hX^1_{T},\Law(\hX^1_T|\tF_T)) - g_{x}(\hX^2_{T},\Law(\hX^2_T|\tF_T)), \nonumber
 \end{gather}
so we get by \ito's lemma that
$$ d(\Delta \hY_{t}\Delta \hX_{t}) = - (\Delta \hY_{t})^2 dt + \Delta \hX_{t} d\Delta \hY_{t} $$
Taking integral from $\tau$ to $T$ and taking expectation conditional on $\tF_\tau$ (where $\tmb{E}_\tau[\cdot] := \mb{E}[\cdot | \tF_\tau]$), we have
\begin{align*}
\int_{\tau}^{T} \tmb{E}_{\tau}[\Delta \hY^{2}_{t}] dt	&\leq  \tmb{E}_\tau[\Delta \hX_\tau \Delta \hY_\tau] - \tmb{E}_{\tau} (\Delta \hY_{T}\Delta \hX_{T})  \\
&=  \tmb{E}_\tau[\Delta \hX_\tau \Delta \hY_\tau]- \tmb{E}_{\tau} ((g_{x}(\hX^1_{T},\Law(\hX^1_T|\tF_T))-g_{x}(\hX^2_{T},\Law(\hX^2_T|\tF_T)))\Delta \hX_{T})
\end{align*}
Using \ref{a4}, we get
$$  \tmb{E}_{\tau} \l[ (g_{x}(\hX^1_{T},\Law(\hX^1_T|\tF_T))-g_{x}(\hX^2_{T},\Law(\hX^2_T|\tF_T)))\Delta \hX_{T}\r] \geq 0 $$
so that, by Young's inequality, 
\begin{align}
\int_{\tau}^{T} \tmb{E}_{\tau}[\Delta \hY^{2}_{t}] dt	&\leq  \tmb{E}[\Delta \hX_\tau \Delta \hY_\tau] \nonumber\\
& \leq \frac{1}{4\eps} \tmb{E}[\Delta \eta^2] + \eps \tmb{E}[\Delta \hY_\tau^2]  \label{b1}
\end{align}
for any $\eps > 0$. \del{Applying \ito's  lemma to $(\Delta \hY_t)^2$ yields,}
\dell{$$ \tmb{E}_\tau[\Delta \hY_t^2] = \tmb{E}_\tau[\Delta \hY_\tau^2] +\int_\tau^t \tmb{E}_\tau[ \Delta \hZ_t^2 + \Delta \htZ_t^2] dt  \geq  \tmb{E}_\tau[\Delta \hY_\tau^2]$$}
\add{Since $\Delta \hY_t$ is a martingale, it follows that $ \tmb{E}_\tau[\Delta \hY_\tau^2] \leq \tmb{E}_\tau[\Delta \hY_t^2] $  }for all $t \in [\tau,T]$. Thus, together with the fact that $\tau \leq T-\frac{\gamma}{2}$, we have
\begin{equation}\label{bY}  
\tmb{E}_\tau[\Delta \hY_\tau^2] \leq \frac{1}{T-\tau}\int_\tau^T \tmb{E}_\tau[\Delta \hY_t^2] dt \leq \frac{2}{\gamma}\int_\tau^T \tmb{E}_\tau[\Delta \hY_t^2] dt
\end{equation}
Similarly, we have
\begin{equation}\label{bYY}  
\tmb{E}_\tau[\Delta Y_\tau^2] \leq \frac{1}{T-\tau}\int_\tau^T \tmb{E}_\tau[\Delta Y_t^2] dt \leq \frac{2}{\gamma}\int_\tau^T \tmb{E}_\tau[\Delta Y_t^2] dt
\end{equation}
\add{The inequality \eqref{bYY} will be used later in the proof. Next, }going back to (\ref{b1}) and use (\ref{bY}), it follows that
$$ \int_{\tau}^{T} \tmb{E}_{\tau}[\Delta \hY^{2}_{t}] dt \leq \frac{1}{4\eps}\tmb{E}_\tau[\Delta \eta^2] + \frac{2\eps}{\gamma} \int_{\tau}^{T} \tmb{E}_{\tau}[\Delta \hY^{2}_{t}] dt $$
Therefore, for sufficiently small $\eps$, we have
$$ \int_{\tau}^{T} \tmb{E}_{\tau}[\Delta \hY^{2}_{t}] dt \leq C_{\eps,\gamma}\tmb{E}_\tau[ \Delta \eta^2] $$
Thus, by (\ref{fbsdeDelta}), 
\begin{equation}\label{boundRef}
	\tmb{E}_\tau(\Delta \hX_T)^2 =  \tmb{E}_\tau\l( \Delta \eta - \int_\tau^T \hY_t dt \r)^2 \leq C'_{\eps,\gamma} \tmb{E}_\tau[\Delta \eta]^2 
\end{equation}
Now we go back to showing (\ref{bddum}). Note that $(\Delta X_{t},\Delta Y_{t}, \Delta Z_t, \Delta \tZ_t)_{\tau \leq t \leq T}$ satisfies
\begin{gather*}
 d\Delta X_{t} = - \Delta Y_{t} dt  \\
 d\Delta Y_{t} = \Delta Z_{t} dW_{t} + \Delta \tZ_{t} \tdW_{t} \\
 \Delta X_{\tau} = \Delta x, \quad \Delta Y_{T} = g_{x}(X^1_{T},\Law(\hX^1_{T}|\tF_T)) - g_{x}(X^2_{T},\Law(\hX^2_T|\tF_T)) 
 \end{gather*}
so we have
\begin{equation}\label{bX}
  \tmb{E}_\tau(\Delta X_T)^2 =  \tmb{E}_\tau\l(\Delta x - \int_\tau^T Y_t dt \r)^2 \leq 2(\Delta x)^2 +  2T \int_\tau^T \mb{E}(\Delta Y_s)^2 ds,
 \end{equation}
 then following the same sequence of inequalities as we did earlier, 
\begin{align*}
 \int_{\tau}^{T}  \tmb{E}_\tau[\Delta Y^{2}_{t}] dt &=   \tmb{E}_\tau(\Delta X_\tau \Delta Y)-  \tmb{E}_\tau \l[ (g_{x}(X^1_{T},\Law(\hX^1_T|\tF_T))-g_{x}(X^2_{T},\Law(\hX^2_T|\tF_T)))\Delta X_{T}\r] \\
&=  \tmb{E}_\tau(\Delta X_\tau \Delta Y_\tau) -  \tmb{E}_\tau \l[ (g_{x}(X^1_{T},\Law(\hX^1_T|\tF_T))-g_{x}(X^2_{T},\Law(\hX^1_T|\tF_T))\Delta X_T\r] \\
& \qquad +\tmb{E}_\tau\l[ g_{x}(X^2_{T},\Law(\hX^2_T|\tF_T))-g_{x}(X^2_{T},\Law(\hX^1_T|\tF_T)))\Delta X_{T}) \r]
\end{align*}
By convexity assumption \ref{a2}, 
$$ \tmb{E}_\tau \l[ (g_{x}(X^1_{T},\Law(\hX^1_T|\tF_T))-g_{x}(X^2_{T},\Law(\hX^1_T|\tF_T)))\Delta X_T\r]  \geq 0 $$
and by Lipschitz continuity (in $m$) assumption \ref{a3}, 
$$ \tmb{E}_\tau\l[ g_{x}(X^2_{T},\Law(\hX^2_T|\tF_T))-g_{x}(X^2_{T},\Law(\hX^1_T|\tF_T)))\Delta X_{T}) \r]
 \leq C_g \tmb{E}_\tau \l((\tmb{E}_T[\Delta \hX_T^2 ])^{\frac{1}{2}}\Delta X_{T}\r) $$
 where $\tmb{E}_T[\cdot] := \mb{E}[\cdot | \tF_T]$. 
Thus, 
\begin{align*}
\int_{\tau}^{T}  \tmb{E}_\tau[\Delta Y^{2}_{t}] dt  
&\leq   \tmb{E}_\tau(\Delta X_\tau \Delta Y_\tau) +C_g  \tmb{E}_\tau ((\tmb{E}_T[\Delta \hX_T^2])^{\frac{1}{2}}\Delta X_{T}) \\
&\leq C_{\eps,C_g}(\Delta x^2 + \tmb{E}_\tau[\Delta \hX_T^2]) + \eps\l(  \tmb{E}_\tau[\Delta Y_\tau^2]+  \tmb{E}_\tau[\Delta X_T^2] \r)
\end{align*}
From  (\ref{bYY}), (\ref{boundRef}), and (\ref{bX}), it follows that 
\begin{align*}
\int_{\tau}^{T}  \tmb{E}_\tau[\Delta Y^{2}_{t}] dt  
&\leq C'_{\eps,C_g}(\Delta x^2 + \mb{E}[\Delta \hX_T^2]) + \eps\l(  \tmb{E}_\tau[\Delta Y_\tau^2]+ 2T\int_\tau^T  \tmb{E}_\tau[\Delta Y_s^2] ds \r) \\
&\leq C'_{\eps,\gamma}(\Delta x^2 + \tmb{E}_\tau[\Delta \eta ^2]) + \eps \l(\frac{2}{\gamma}+2T \r)\int_\tau^T  \tmb{E}_\tau[\Delta Y_s^2] ds 
\end{align*}
Thus, by selecting $\eps$ sufficiently small, we get
$$  \int_\tau^T  \tmb{E}_\tau[\Delta Y_s^2] ds \leq C_{\gamma,T,C_g}(\Delta x^2 + \tmb{E}_\tau[ \Delta \eta ^2]) $$
Therefore, using (\ref{bYY}) again, 
$$  \tmb{E}_\tau[\Delta Y_\tau^2]  \leq \frac{2}{\gamma} \int_\tau^T  \tmb{E}_\tau[\Delta Y_s]^2 ds \leq  \tilde{C}_{\gamma,T,C_g}(\Delta x^2 + \tmb{E}_\tau[\Delta \eta ]^2) $$
Observe that all the coefficients in the FBSDE (\ref{fbsdeInd}) are $\tF_t$-progressively measurable, therefore, $Y_\tau$ is $\tF_\tau$-measurable. Hence, 
$$ \Delta Y_\tau^2 =  \tmb{E}_\tau[\Delta Y_\tau^2]  \leq   \tilde{C}_{\gamma,T,C_g}(\Delta x^2 + \tmb{E}_\tau[\Delta \eta ]^2)  $$
or equivalently,
$$ (u(x_1,\eta_1)-u(x_1,\eta_2))^2 \leq \tilde{C}_{\gamma,T,C_g}((x_1-x_2)^2 + \mb{E}[ (\eta_1-\eta_2)^2 | \tF_\tau] ) $$
as desired.

\subsection{Proof of Lemma \ref{phi2}} \label{fix} 

 Let $\ha^1,\ha^2 \in \H^2([s,\tau];\R)$, $(X^i_t,Y^i_t,Z^i_t,\tZ^i_t)_{s \leq t \leq \tau}$ denote the solution of corresponding FBSDE (\ref{fbsde_phi}) from the definition of $\Phi^{s,\tau,\eta,u}(\ha^i)$, $(\hX_t^i)_{s \leq t \leq \tau}$ denote the state process corresponding to $(\ha^i_t)_{s \leq t \leq \tau}$ with initial $\eta \in \Ltwo_{\F_s}$ as defined by (\ref{stateProcess}). Let $\Delta X_{t} = X^1_{t}-X^2_{t}$ and define $\Delta Y_t, \Delta Z_t, \Delta \tZ_t, \Delta \ha_t, \Delta \hX_t$ similarly. Note that $(\Delta X_{t},\Delta Y_{t}, \Delta Z_t, \Delta \tZ_t)_{s \leq t \leq \tau}$ satisfies
\begin{gather*}
 d\Delta X_{t} = - \Delta Y_{t} dt  \\
 d\Delta Y_{t} = \Delta Z_{t} dW_{t} + \Delta \tZ_{t} \tdW_{t} \\
 \Delta X_{s} = 0, \quad \Delta Y_{\tau} = u(X^1_{\tau},\hX^1_{\tau}) - u(X^2_{\tau},\hX^2_\tau) 
 \end{gather*}
 \del{By applying \ito's lemma to $\Delta Y_t^2$, we get}
 \dell{$$ \mb{E}[\Delta Y_t^2] = \mb{E}[\Delta Y_\tau^2] -\int_t^\tau \mb{E}[ \Delta Z_t^2 + \Delta \tZ_t^2] dt  \leq  \mb{E}[\Delta Y_\tau^2]$$}
\add{Since $\Delta Y_t$ is a martingale, it follows that $ \mb{E}[\Delta Y_t^2] \leq \mb{E}[\Delta Y_\tau^2]$ for all $t \in [s,\tau]$ and} thus,
\begin{equation}\label{bdx}
 \mb{E}[\Delta X_\tau^2] = \mb{E}\l[ \l( \int_s^\tau \Delta Y_t dt \r)^2 \r]  \leq (\tau-s) \int_s^\tau \mb{E}[\Delta Y_t^2] dt \leq (\tau-s)^2 \mb{E}[\Delta Y_\tau^2] 
 \end{equation}
Furthermore,
\begin{equation}\label{bdhx}
 \mb{E}[\Delta \hX_\tau]^2  = \mb{E}\l[ \l( \int_s^\tau \Delta \ha_t dt \r)^2 \r] \leq (\tau-s) \mb{E}\l( \int_s^\tau \Delta \ha_t^2 \r) dt 
\end{equation}
Therefore, by (\ref{bddum}), (\ref{bdx}), and (\ref{bdhx}), we have
\begin{align*}
	\mb{E}[\Delta Y_\tau^2]  &= \mb{E}[ u(X^1_{\tau},\hX^1_\tau) - u(X^2_{\tau},\hX^2_\tau) ] \\
						&\leq C_u\l( \mb{E}[\Delta X_\tau^2] + \mb{E}[\mb{E}[ \Delta \hX_\tau^2 | \tF_\tau] ] \r) \\
						&\leq  C_u \l[  (\tau-s)^2\mb{E}[\Delta Y_\tau^2] + (\tau-s)\mb{E}\l( \int_s^\tau (\Delta \ha_t)^2 dt \r)  \r]
\end{align*}
Hence, for $\tau-s$ sufficiently small such that $C_u(\tau-s)^2 < 1$, we get
$$ \mb{E}[\Delta Y_\tau^2] \leq \frac{C_u(\tau-s)}{1-C_u(\tau-s)^2}\mb{E}\l( \int_s^\tau (\Delta \ha_t)^2 dt \r) =\frac{C_u(\tau-s)}{1-C_u(\tau-s)^2} \| \Delta \ha \|^2_{\H^2([s,\tau];\R)} $$
and it follows that
$$ \| \Phi(\ha^1)-\Phi(\ha^2) \|^2_{\H^2([s,\tau];\R)} =  \int_s^\tau \mb{E}[\Delta Y_t^2] dt  \leq (\tau-s)\mb{E}[\Delta Y_\tau^2] \leq  \frac{C_u(\tau-s)^2}{1-C_u(\tau-s)^2} \| \ha^1-\ha^2 \|^2_{\H^2([s,\tau];\R)} $$
As a result, we get a contraction map for sufficiently small $\tau-s$ depending only on $C_u$ as desired.

\subsection{Proof of Lemma \ref{alpha_estimate}}

 Applying \ito's lemma to $\hX_t\hY_t$ from $\tau$ to $T$ yields
$$ \mb{E}[\hX_T g_x(\hX_T,\Law(\hX_T|\tF_T))] - \mb{E}[\xi \hY_\tau] = - \int_\tau^T \mb{E}[\hY_t^2] dt + \int_\tau^T \mb{E}[ \sigma \hZ_t + \tsg \htZ_t] dt $$
From the weak monotonicity assumption \ref{a4}, we get
$$\mb{E}[\hX_T g_x(\hX_T,\Law(\hX_T|\tF_T))]  \geq \mb{E}[\hX_T g_x(0,\delta_0)] $$
Thus,
\begin{align*}
\int_\tau^T \mb{E}[\hY_t^2] dt & \leq \mb{E}[\xi \hY_\tau] - \mb{E}[ \hX_T g_x(0,\delta_0) ] +  \int_\tau^T \mb{E}[ \sigma \hZ_t + \tsg \htZ_t] dt \\
	&\leq C_{\eps,T} \l[  \mb{E}[\xi^2] + g^2_x(0,\delta_0) + \sigma^2 + \tsg^2 \r] + \eps \mb{E}\l[\hX_T^2 + \hY_\tau^2 + \int_\tau^T [\hZ_r^2 +  \htZ_r^2]dr \r] 
\end{align*}
Applying \ito's lemma to $\hY_t^2$ from $t$ to $T$ together with the Lipschitz assumption \ref{a1} and \ref{a3} give, for all $t \in [\tau,T]$, 
\begin{equation}\label{bound_y2} 
\mb{E}[\hY_t^2] + \int_t^T \mb{E}[  \hZ_r^2 +  \htZ_r^2]dr \leq \mb{E}[\hY_T^2] = \mb{E}[g^2_x(\hX_T,\Law(\hX_T|\tF_T))] \leq C_g(g^2_x(0,\delta_0) + \mb{E}[\hX_T^2]) 
\end{equation}
Thus, we get a bound of the form
$$ \int_\tau^T \mb{E}[\hY_t^2] dt  \leq C_{\eps,T,C_g}(\mb{E}[\xi^2] + g^2_x(0,\delta_0) + \sigma^2+\tsg^2) + \eps\mb{E}[\hX_T^2] $$
Note that by Doob's Martingale inequality, we also have
$$ \mb{E}[\sup_{\tau\leq t \leq T} |\hX_t|^2 ] \leq C_T\l( \int_\tau^T \mb{E}[\hY_t^2] dt  + \sigma^2 + \tsg^2 \r) $$
Combining with the previous bound, we can select $\eps$ sufficiently small depending only on $T,C_g$ to get 
\begin{equation}\label{bound_x}
  \mb{E}[\sup_{\tau\leq t \leq T} |\hX_t|^2 ] \leq C_{T,C_g}(\mb{E}[\xi^2] + g^2_x(0,\delta_0) + \sigma^2+\tsg^2) 
  \end{equation}
Applying Doob's Martingale inequality to $Y_t$ and use (\ref{bound_y2}) and (\ref{bound_x}), we then complete the proof.

\section{Summary and Conclusions}\label{summary}
	
	The aim of this paper is to establish an existence and uniqueness result for a Mean Field Game in the presence of common noise. We assume a linear-convex control problem with the quadratic running cost depending only on the control and a general convex and Lipschitz terminal cost function. In addition, we assume that the terminal cost also satisfies what we call a weak monotonicity condition (see (\ref{monotone})).
	
	For the existence result, our assumptions are similar to those used in \cite{Carmona1} for a $0$-MFG problem except that we replace the weak mean-reverting assumption (\ref{weakMR}) with the weak monotonicity condition (\ref{monotone}). Our existence proof is also different from those of a $0$-MFG problem which rely on the Schauder fixed point theorem. Instead, we apply the Banach fixed point theorem to given existence of a solution over a small time duration and show that it can be extended to an arbitrary time duration. For the uniqueness part, under the convexity assumption on the cost functions, our result extends those of a $0$-MFG problem as our weak monotonicity condition is weaker than the monotonicity condition used in \cite{cardaliaguet2010,Carmona1}. 
	
	Although we assume a simple state process and running cost, our result is expected to hold for a more general setting, namely a model with a linear state process
	$$ dX_t = (A\alpha_t + BX_t + C\bar{m}_t) dt + \sigma dW_t + \tsg d\tW_t$$
and a running cost of the form
$$ f(t,x,\alpha,m) = f_0(t,x,\alpha)+f_1(t,x,m) $$
where $f_0$ is Lipschitz and convex in $(x,\alpha)$ and strongly convex in $\alpha$, $f_1$ is Lipschitz in $(x,m)$ and satisfies the weak monotonicity condition.


\bibliographystyle{abbrv}
\bibliography{../MFG}

\end{document}